\documentclass[12pt]{article} 
\usepackage[utf8]{inputenc} 
\usepackage{lmodern}
\usepackage[english]{babel}
\usepackage[margin=1.1in]{geometry} 
\usepackage{amsmath, amsthm, amssymb}
\usepackage{graphicx} 
\usepackage{booktabs} 
\usepackage{array} 
\usepackage{color} 
\usepackage{seqsplit}
\usepackage{paralist} 
\usepackage{verbatim} 
\usepackage{subfig} 

\newtheorem{theorem}{Theorem}
\newtheorem{prop}{Proposition}
\newtheorem{defn}{Definition}
\newtheorem{exmp}{Example}[section]
\newtheorem{rmk}{Remark}[section]

\title{Two divisors of $(n^2+1)/2$ summing up to $\delta n + \delta \pm 2$, $\delta$ even}
\author{Sanda Buja\v{c}i\'{c} Babi\'c\footnote{sbujacic@math.uniri.hr, +385 51 584 654. The author is supported by Croatian Science Foundation grant number 6422.} \\ Department of Mathematics,\\ University of Rijeka, Croatia}
\date{}

\begin{document}

\maketitle

\begin{abstract}
We prove there exist infinitely many odd integers $n$ for which there exists a pair of positive divisors $d_1, d_2>1$ of $(n^2+1)/2$ such that $$d_1+d_2=\delta n+(\delta+2).$$
We prove the similar result for $\varepsilon=\delta-2$ and $\delta\equiv4, 6\pmod{8}$ using different approaches and methods.	
\end{abstract}

\begin{section}{Introduction}
A polynomial $f\in\mathbb{C}[x]$ is called \emph{indecomposable} over $\mathbb{C}$ if $$f=g\circ h, \ \ g,h\in\mathbb{C}[x]$$ implies $\deg g=1$ or $\deg h=1$. Ayad \cite{ayad0} gives sufficient conditions for a polynomial $f$ to be indecomposable in terms of its critical points and critical values and conjectures there do not exist two divisors $d_1, d_2>1$ of $(p^2+1)/2$ such that $$d_1+d_2=p+1,$$ $p$ prime.

Ayad and Luca \cite{ayad} deal with a similar, but more general problem and prove that there does not exist an odd integer $n>1$ and two positive divisors $d_1, d_2$ of $(n^2+1)/2$ such that \begin{equation}\label{eq:sazetak1}
d_1+d_2=n+1.
\end{equation}

Finally, Dujella and Luca \cite{dujella} replace the linear polynomial $n+1$ in (\ref{eq:sazetak1}) with an arbitrary linear polynomial $\delta n+\varepsilon$  for $\delta>0$  and $\varepsilon$ given integers and try to answer the question wheather there exist infinitely many positive and odd integers $n$ for which there are two divisors $d_1, d_2$ of $(n^2+1)/2$ such that $$d_1+d_2=\delta n+\varepsilon.$$ The problem can be separated in two cases: it is either $\delta\equiv\varepsilon\equiv1\pmod{2}$ or $\delta\equiv\varepsilon+2\equiv0 \enspace \textrm{or}\enspace 2\pmod{4}$. In \cite{dujella} the authors deal with the case $$\delta\equiv\varepsilon\equiv1\pmod{2}.$$ 

Buja\v{c}i\'{c} Babi\'c \cite{bujacic} deals with the case $$\delta\equiv\varepsilon+2\equiv0 \enspace \textrm{or}\enspace 2\pmod{4},$$ for some fixed $\delta$ or $\varepsilon$. 
\newline\indent One--parametric families of even coefficients $\delta$ and $\varepsilon$ of the linear polynomial $\delta n +\varepsilon$ where $$\varepsilon = \delta \pm 2$$ are discussed in this article. We prove the existence of infinitely many odd integers $n$ for which there exists a pair of positive divisors $d_1, d_2>1$ of $(n^2+1)/2$ such that $$d_1+d_2=\delta n+(\delta+2).$$
We prove the similar result for $\varepsilon=\delta-2$ and $\delta\equiv4, 6\pmod{8}$ using different approaches and methods while the same problem for $\delta\equiv0,2\pmod{8}$ still remains open.
\end{section}

\begin{section}{$d_1+d_2=\delta n+\varepsilon$ for $\varepsilon= \delta+2$}

\begin{theorem}
There are infinitely many odd positive integers $n$ for which there exist divisors $d_1, d_2>1$ of $(n^2+1)/2$ such that $$d_1+d_2=\delta(n+1)+2$$ for every even positive integer $\delta$.
\end{theorem}
\begin{proof}
	
Let the integer $\delta$ be even and positive, integer $n$ positive odd and $d_1, d_2>1$ divisors of $(n^2+1)/2$ such that $$d_1+d_2=\delta(n+1)+2.$$
As in \cite{bujacic}, let $g=\textrm{gcd}(d_1, d_2)$. There exists an integer $d$ such that $$d_1d_2=\frac{g(n^2+1)}{2d}.$$
From the identity $$(d_2-d_1)^2=(d_1+d_2)^2-4d_1d_2,$$ we easily get
\begin{equation}\label{medujedn}
d(d_2-d_1)^2=(\delta^2d-2g)n^2+2d\delta(\delta+2)n+\delta^2d+4d\delta +4d-2g.
\end{equation}
Multiplying (\ref{medujedn}) by $\delta^2d-2g,$ it is obtained
$$d(\delta^2d-2g)(d_2-d_1)^2=$$
$$=(\delta^2d-2g)^2n^2+2d\delta(\delta^2d-2g)(\delta+2)n+\delta^4d^2+4\delta^3d^2+4d^2\delta^2-4\delta^2dg-8d\delta g-8dg+4g^2.$$
After introducing the supstitutions $X=(\delta^2d-2g)n+d\delta(\delta+2), \ Y=d_2-d_1$ the previous equation becomes \begin{equation}\label{section:20}
X^2-d(\delta^2d-2g)Y^2=4\delta^2dg+8d\delta g+8dg-4g^2.\end{equation}

For $d=g$ the right--hand side of $(\ref{section:20})$ becomes a perfect square, so the equation of the form $$X^2-d(\delta^2d-2d)Y^2=(2d(\delta+1))^2$$ is obtained.
After introducing $X=dX'$, we get
\begin{equation}\label{section:21}
X'^2-(\delta^2-2)Y^2=4(\delta+1)^2.
\end{equation} 
For $\delta$ even, $\delta^2-2\equiv2\pmod{4}$ is never a perfect square, so $(\ref{section:21})$ is a pellian equation. If we introduce $X'=2(\delta+1)U$ and $Y=2(\delta+1)V$ and divide (\ref{section:21}) by $(2(\delta+1))^2$, we get 
\begin{equation}\label{section:22}
U^2-(\delta^2-2)V^2=1,
\end{equation}
which is a Pell equation and has infinitely many solutions $(U, V)$. Consequently, the pellian equation $(\ref{section:21})$ has infinitely many solutions $(X', Y)$. 
Since the continuous fraction expansion of $\sqrt{\delta^2-2}$ is $$\sqrt{\delta^2-2}=[\delta-1; \overline{1, \delta-2, 1, 2\delta-2}]$$ 
the fundamental solution of (\ref{section:22}) is 
$(U_1, V_1)=(\delta^2-1, \delta)$. All the solutions of (\ref{section:22}) are generated using the recurrence relations of the form
\begin{equation}\label{eq:dodanarek0}U_0=1, \hspace{5pt} U_1=\delta^2-1,\hspace{5pt} U_{m+2}=2(\delta^2-1)U_{m+1}-U_m,\end{equation}
\begin{equation}\label{eq:dodanarek}
V_0=0,\hspace{5pt} V_1=\delta,\hspace{5pt} V_{m+2}=2(\delta^2-1)V_{m+1}-U_m, \hspace{5pt} m\in\mathbb{N}_0.\end{equation}
From $X=2d(\delta+1)U$ and $X=(\delta^2d-2d)n+d\delta(\delta+2)$, it is easily obtained \begin{equation}\label{eq:trecin}
n=\frac{2(\delta+1)U-\delta(\delta+2)}{\delta^2-2}.\end{equation}
Now, we show that (\ref{eq:trecin}) are odd integers. We get $$U_0=1\equiv1\pmod{(\delta^2-2)}, \enspace U_1=\delta^2-1\equiv1\pmod{(\delta^2-2)}.$$ Assume that the congruences $U_{m-1}\equiv U_m\equiv1\pmod{(\delta^2-2)}$ hold. We get 
$$U_{m+1}=2(\delta^2-1)U_m-U_{m-1}\equiv2-1\equiv1\pmod{(\delta^2-2)}.$$
Hence, $$2(\delta+1)U-\delta(\delta+2)\equiv2\delta+2-\delta^2-2\delta\equiv-(\delta^2-2)\equiv0\pmod{(\delta^2-2)}$$ which proves that $n\in\mathbb{Z}$. 
Because $\delta$ is even and $U$ odd, we conclude $$2(\delta+1)U-\delta(\delta+2)\equiv2U\equiv2\pmod{4} \enspace\textnormal{and}\enspace \delta^2-2\equiv2\pmod{4}.$$ Consequently, every integer $n$ of the form (\ref{eq:trecin}) is odd.
\end{proof}

\begin{exmp}
\normalfont We generate integers $n, d_1$ and $d_2$ for each solution $(U_i, V_i), \ i\in\mathbb{N},$ of the equation (\ref{section:22}).
Because of
	\begin{equation}\label{eq:savjet}
	d_1+d_2=\delta n+\delta+2, \enspace d_1d_2=\frac{n^2+1}{2},
	\end{equation}
	$d_1$ and $d_2$ can be interpreted as solutions of the quadratic equation. Using Vieta's formulas we determine expressions for $d_1, d_2$ for each positive odd integer $n$. Namely, $d_1$ and $d_2$ are solutions of the quadratic equation of the form
	$$t^2-(d_1+d_2)t+d_1d_2=0,$$
	\begin{equation}\label{eq:kvadratna}
	t^2-(\delta n+\delta +2)t+\frac{n^2+1}{2}=0,
	\end{equation}
	where $d_i=t_i,\ i=1, 2$.
	From $(\ref{eq:kvadratna})$ the following is obtained
	\begin{equation}\label{eq:formula}
	t_{1, 2}=\frac{2\delta n+2\delta+4\pm\sqrt{4(\delta n+\delta+2)^2-8(n^2+1)}}{4}.
	\end{equation}
	For $U=U_1=\delta^2-1$, we get $$n=\frac{2(\delta+1)(\delta^2-1)-\delta(\delta+2)}{\delta^2-2}=2\delta+1.$$ 
	Introducing $n=2\delta+1$ into $(\ref{eq:formula})$, we obtain
	$$t_{1,2}=\frac{2\delta  (2\delta+1)+2\delta+4\pm4\delta(1+\delta)}{4}=\frac{4\delta^2+4\delta+4\pm(4\delta^2+4\delta)}{4},$$
	so $$t_1=1, \enspace t_2=2\delta^2+2\delta+1.$$
	For $U=U_2=2\delta^4-4\delta^2+1$, we have $$n=\frac{2(\delta+1)(2\delta^4-4\delta^2+1)-\delta(\delta+2)}{\delta^2-2}=4\delta^3+4\delta^2-1.$$ 
	Introducing $n=4\delta^3+4\delta^2-1$ into $(\ref{eq:formula})$, we get
	$$t_{1, 2}=\frac{8\delta^4+8\delta^3+4\pm8\delta(\delta-1)(\delta+1)^2}{4},$$
	hence $$t_1=2\delta^2+2\delta+1, \enspace t_2=4\delta^4+4\delta^3-2\delta^2-2\delta+1.$$
	Analogously, for $U=U_3=4\delta^6-12\delta^4+9\delta^2-1$, we obtain $$n=\frac{2(\delta+1)(4\delta^6-12\delta^4+9\delta^2-1)-\delta(\delta+2)}{\delta^2-2}=8\delta^5+8\delta^4-8\delta^3-8\delta^2+2\delta+1,$$ and $$t_1=4\delta^4+4\delta^3-2\delta^2-2\delta+1,\enspace t_2=8\delta^6+8\delta^5-12\delta^4-12\delta^3+4\delta^2+4\delta+1.$$

Consequently, we generate infinitely many triples $(n, d_1, d_2)$ of the form
	\[
	\left\{ 
	\begin{array}{l l}
	n=2\delta+1,\\
	d_1=1,\\
	d_2=2\delta^2+2\delta+1.\\
	\end{array} \right.
	\]
	\[
	\left\{ 
	\begin{array}{l l}
	n=4\delta^3+4\delta^2-1 ,\\
	d_1=2\delta^2+2\delta+1,\\
	d_2= 4\delta^4+4\delta^3 - 2\delta^2 - 2\delta + 1.\\
	\end{array} \right.
	\]
	\[
	\left\{ 
	\begin{array}{l l}
	n=8\delta^5 + 8\delta^4 - 8\delta^3 - 8\delta^2 + \delta+ 1,\\
	d_1= 4\delta^4+ 4\delta^3- 2\delta^2- 2\delta+  1,\\
	d_2= 8\delta^6 - 12\delta^4 - 12\delta^3+ 4\delta^2+ 4\delta +  1.\\
	\end{array} \right.
	\] \\
\end{exmp}
\vspace{-25pt}
It is easy to notice that the divisor $d_2$ of $\frac{n^2+1}{2}$ for $n=2\delta+1$ is the divisor $d_1$ of $\frac{n^2+1}{2}$ for $n=4\delta^3+4\delta^2-1$. The divisor $d_2$ of $\frac{n^2+1}{2}$ for $n=4\delta^3+4\delta^2-1$ is the divisor $d_1$ of $\frac{n^2+1}{2}$ for $n=8\delta^5 + 8\delta^4 - 8\delta^3 - 8\delta^2 + \delta+ 1$, etc. 

The quadratic equations of the form  $(\ref{eq:kvadratna})$ that are generated using two integers $n$ formed by two consecutive terms of the recursion sequence $U_m, \ m\geq1$ have a mutual root. We prove that claim.

\begin{defn}
	\textit{Let $f, g\in K[x]$ be two polynomials of the degrees $l$ and $m$, respectively, with coefficients in an arbitrary field $K$, hence
		$$f(x)=a_0x^l+\dots+a_l, \enspace a_0\neq0, \enspace l>0,$$}
	\vspace{-15pt}
	\begin{equation}\label{eq:res01}
	g(x)=b_0x^m+\dots+b_m, \enspace b_0\neq0, \enspace m>0.
	\end{equation}
	The \emph{resultant} \textit{of $f$ and $g$, $\textnormal{Res}(f, g)$, is the determinant $(l+m)\times(l+m)$} of the form\\
	\begin{center}
		$\textnormal{Res}(f, g)=\textrm{det}
		\begin{bmatrix}
		a_0 &  &  &  & b_0 &  &  & \\
		a_1 & a_0 &  &  & b_1 & b_0 & & \\
		a_2 & a_1 & \ddots &  & b_2 & b_1 & \ddots & \\
		\vdots & a_2 & \ddots & a_0 & \vdots & b_2 & \ddots & b_0\\
		a_l & \vdots & \ddots & a_1 & b_m & \vdots & \ddots & b_1\\
		& a_l  &  & a_2 & & b_m & & b_2\\
		&  & \ddots & \vdots & & & \ddots & \vdots\\
		&  &  & a_l & & & & b_m\\
		\end{bmatrix},$\\
	\end{center}
	\textit{where empty spaces stand for zeros.}
\end{defn}

Two polynomials $f, g$ have a common root if and only if resultant $\textnormal{Res}(f, g)=0$.  In our case, the associated quadratic polynomial for $n=2\delta+1$ is
\begin{equation}\label{eq:kv1}
f_1(t)=2t^2-2(2\delta^2+2\delta+2)t+4\delta^2+4\delta+2.\end{equation}
Analogously, for $n=4\delta^3+4\delta^2-1$, we get
\begin{equation}\label{eq:kv2}
f_2(t)=2t^2-2(4\delta^4+4\delta^3+2)t+16\delta^6+32\delta^5+16\delta^4-8\delta^3-8\delta^2+2.
\end{equation}
\noindent The polynomials (\ref{eq:kv1}) and (\ref{eq:kv2}) have one mutual root which implies $\textnormal{Res}(f_1, f_2)=0.$ More precisely,

\begin{scriptsize}
	\begin{center}
		$\begin{vmatrix}
		2 &  0 & 2 & 0\\
		-2(2\delta^2+2\delta+2) & 2 & -2(4\delta^4+4\delta^3+2) & 2\\
		4\delta^2+4\delta+2 & -2(2\delta^2+2\delta+2)& 16\delta^6+32\delta^5+16\delta^4-8\delta^3-8\delta^2+2 & -2(4\delta^4+4\delta^3+2)\\
		0 & 4\delta^2+4\delta+2 & 0 & 16\delta^6+32\delta^5+16\delta^4-8\delta^3-8\delta^2+2\\
		\end{vmatrix}=0.$\\
	\end{center}
\end{scriptsize}

This property holds generally.

\begin{prop}
	\textit{Two quadratic polynomials of the form (\ref{eq:kvadratna}) generated by two integers $n$ determined by two consecutive terms of the recursive sequence $U_m,\enspace m\geq1$, have a mutual root. In other words, for every two integers $n$ that are generated by two consecutive terms of the recursive sequence $U_m,\enspace m\geq1$, the corresponding integers $\frac{n^2+1}{2}$ have a common divisor.}
\end{prop}
\noindent \textit{Proof}.

Let $U_{m-1},\enspace U_m$ be two consecutive terms of the recursive sequence (\ref{eq:dodanarek0}) that we use for generating two integers $n$ of the form (\ref{eq:trecin}). For each such integer $n$, from (\ref{eq:kvadratna}), we get two quadratic polynomials of the form:
\begin{equation}\label{eq:kv3}
f_{m-1}(t)=2t^2-2\left(\delta\frac{2(\delta+1)U_{m-1}-\delta(\delta+2)}{\delta^2-2}+\delta+2\right)t+\left(\frac{2(\delta+1)U_{m-1}-\delta(\delta+2)}{\delta^2-2}\right)^2+1,
\end{equation}
\begin{equation}\label{eq:kv4}
f_m(t)=2t^2-2\left(\delta\frac{2(\delta+1)U_{m}-\delta(\delta+2)}{\delta^2-2}+\delta+2\right)t+\left(\frac{2(\delta+1)U_{m}-\delta(\delta+2)}{\delta^2-2}\right)^2+1.
\end{equation}
The resultant $\textnormal{Res}(f_{m-1}, f_m)$ is
\begin{tiny}
	\begin{center}
		$\begin{vmatrix}
		2 &  0 & 2 & 0\\
		-2\left(\delta\frac{2(\delta+1)U_{m-1}-\delta(\delta+2)}{\delta^2-2}+\delta+2\right) & 2 & -2\left(\delta\frac{2(\delta+1)U_{m}-\delta(\delta+2)}{\delta^2-2}+\delta+2\right) & 2\\
		\left(\frac{2(\delta+1)U_{m-1}-\delta(\delta+2)}{\delta^2-2}\right)^2+1 & -2\left(\delta\frac{2(\delta+1)U_{m-1}-\delta(\delta+2)}{\delta^2-2}+\delta+2\right) & \left(\frac{2(\delta+1)U_{m}-\delta(\delta+2)}{\delta^2-2}\right)^2+1 & -2\left(\delta\frac{2(\delta+1)U_{m}-\delta(\delta+2)}{\delta^2-2}+\delta+2\right)\\
		0 & \left(\frac{2(\delta+1)U_{m-1}-\delta(\delta+2)}{\delta^2-2}\right)^2+1 & 0 & \left(\frac{2(\delta+1)U_{m}-\delta(\delta+2)}{\delta^2-2}\right)^2+1\\
		\end{vmatrix}=$\\
	\end{center}
\end{tiny}
$$=\frac{64(1+\delta^4)(U_m-U_{m-1})^2(\delta^4+(U_m+U_{m-1})^2-2\delta^2(1+U_mU_{m-1}))}{(\delta^2-2)^4}.$$
$\textnormal{Res}(f_{m-1}, f_m)=0$ if and only if 
\begin{equation}\label{eq:matind}
\delta^4-2\delta^2(U_mU_{m-1}+1)+(U_m+U_{m-1})^2=0.
\end{equation}
Using mathematical induction on $m$ we show $(\ref{eq:matind})$. By definition, we have $U_0=1$. 
Introducing $U_0, \enspace U_1$ into the previous equation, we get
$$\delta^4-2\delta^2+(2-2\delta^2)(\delta^2-1)+1+(\delta^2-1)^2=\delta^4-2\delta^4-2+2\delta^2+1+\delta^4-2\delta^2+1=0.$$
\noindent We assume $$\delta^4-2\delta^2+(2-2\delta^2)U_{m-1}U_m+U_{m-1}^2+U_m^2=0.$$
The recursive sequence for $U_m$ is $$U_{m+1}=2(\delta^2-1)U_m-U_{m-1}, \enspace m\in\mathbb{N},$$
so, we easily get
$$\delta^4-2\delta^2+(2-2\delta^2)U_m(2(\delta^2-1)U_m-U_{m-1})+(2(\delta^2-1)U_m-U_{m-1})^2+U_m^2=0,$$
\noindent or
$$\delta^4-2\delta^2+(2-2\delta^2)U_{m-1}U_m+U_{m-1}^2+U_m^2=0,$$
which is satisfied according to the hypothesis of the mathematical induction. 
\begin{flushright}
	$\square$
\end{flushright}
\begin{exmp}
\normalfont	

	Let $\delta=8$. We get
	$$(n, \frac{n^2+1}{2}, d_1, d_2, \delta, \varepsilon)=(17, 145, 1, 145, 8, 10), \ (2303, 2651905, 145, 18289, 8, 10),... $$

\noindent	Let $\delta=10$. We get
	$$(n, \frac{n^2+1}{2}, d_1, d_2, \delta, \varepsilon)=(21, 221, 1, 221, 10, 12), \ (4399, 9675601, 221, 43781, 10, 12),... $$
\end{exmp}

\end{section}

\section{$d_1+d_2=\delta n+\varepsilon$ for $\varepsilon=\delta-2$}

\indent In this section, we assume coefficients $\delta$ and $\varepsilon$ of the linear polynomial $\delta n+\varepsilon$ are even and $\varepsilon=\delta-2$. Our goal is to show that there exist infinitely many positive integers $n$ such that two divisors $d_1, d_2$ of $(n^2+1)/2$ satisfy
\begin{equation}\label{eq:svojstvo}
	d_1+d_2=\delta n+\delta-2.
\end{equation} 


Like in the previous section, we set $g=\textnormal{gcd}(d_1, d_2)$. There exists $d\in\mathbb{N}$ such that $$d_1d_2=\frac{g(n^2+1)}{2d}.$$ It is easily obtained $g\equiv d\equiv d_1\equiv d_2\equiv1\pmod{4}$. From the identity $$(d_2-d_1)^2=(d_1+d_2)^2-4d_1d_2,$$
we get the equation
\begin{equation}{\label{eq:210}}X^2-d(d\delta^2-2g)Y^2=2dg(\delta^2+\varepsilon^2)-4g^2,\end{equation}
after introducing the supstitutions of the form $X=n(d\delta^2-2g)+d\delta\varepsilon$ and $Y=d_2-d_1$.

Because $g\mid(\delta^2n^2-\varepsilon^2)$ and $g\mid\delta^2(n^2+1)$, we conclude 
$$g\mid(\delta^2+\varepsilon^2).$$
For integers $\delta, \varepsilon$ even we get $\delta^2+\varepsilon^2\equiv0\pmod{4}$. Generally, because $g\equiv1\pmod{4}$, we conclude $g\mid\frac{\delta^2+\varepsilon^2}{4}$. For $\varepsilon=\delta-2$, we get $$g\hspace{2pt}\big|\hspace{2pt}\frac{\delta^2-2\delta+2}{2}.$$
Setting $g=\frac{\delta^2+\varepsilon^2}{4}=\frac{\delta^2-2\delta+2}{2}$ the equation $(\ref{eq:210})$ becomes 
\begin{equation}\label{eq:220}X^2-d(d\delta^2-2g)Y^2=4g^2(2d-1).\end{equation}
For $d=2k^2-2k+1, \enspace k\in\mathbb{N},$ the right-hand side of the equation $(\ref{eq:220})$ is a prefect square. 
More precisely,
\begin{equation}\label{eq:28}
	X^2-2(2k^2-2k+1)(\delta k-1)(\delta k-\delta +1)Y^2=(2g(2k-1))^2.
\end{equation}
The associated Pell equation of $(\ref{eq:28})$ is
\begin{equation}\label{eq:29}
	U^2-2(2k^2-2k+1)(\delta k-1)(\delta k-\delta +1)V^2=1.
\end{equation}

Because the period length of the continued fraction expansion of $$\sqrt{2(2k^2-2k+1)(\delta k-1)(\delta k-\delta +1)}$$ depends on $k\in\mathbb{N}$, the approach that we have used in the previous section cannot be used again. So, in this case, we have to construct another method  and search for the solutions $(X, Y)=(2g(2k-1)U, \enspace 2g(2k-1)V)$ of the equation (\ref{eq:28}), where $(U, V)$ are solutions of the equation (\ref{eq:29}). The solutions of the equation $(\ref{eq:28})$ have to satisfy the additional condition, 
\begin{equation}\label{eq:lalala}
X\equiv d\delta\varepsilon\equiv d\delta(\delta-2)\pmod{2(\delta k-1)(\delta k-\delta+1)},\end{equation} 
in order to fulfill the request for $n$ to be an integer. 

We set $$a=2k^2-2k+1, \enspace b=\delta k-1, \enspace c=\delta k-\delta+1.$$
The equation $(\ref{eq:29})$  becomes
\begin{equation}\label{eq:uh}
U^2-2abcV^2=1.
\end{equation}
The fundamental solution $(U_0, V_0)$ of that equation satisfies
$$(U_0-1)(U_0+1)=2abcV_0^2.$$
It is easy to conclude $4\mid(U_0-1)(U_0+1)$ and
$V_0$ is even. So, we set $V_0=2st, \enspace s, t\in\mathbb{N}$. The previous equation becomes
$$(U_0-1)(U_0+1)=8abcs^2t^2.$$

If we assume $a, b, c$ are prime, number of factorizations of the equation (\ref{eq:uh}) is the least possible.
In that case and because $a, b, c\neq 2$, we deal only with the following factorizations
\begin{center}
$1^{\pm})\enspace\enspace U_0\pm1=2abcs^2, \hspace{5pt} U_0\mp1=2^2t^2,$\\
$2^{\pm}) \enspace\enspace U_0\pm1=2^2abcs^2, \hspace{5pt} U_0\mp1=2t^2,$\\
$3^{\pm}) \enspace\enspace U_0\pm1=2abs^2, \hspace{5pt} U_0\mp1=2^2ct^2,$\\
$4^{\pm}) \enspace\enspace U_0\pm1=2acs^2, \hspace{5pt} U_0\mp1=2^2bt^2,$\\
$5^{\pm}) \enspace\enspace U_0\pm1=2bcs^2, \hspace{5pt} U_0\mp1=2^2at^2,$\\
$6^{\pm}) \enspace\enspace U_0\pm1=2as^2, \hspace{5pt} U_0\mp1=2^2bct^2,$\\
$7^{\pm}) \enspace\enspace U_0\pm1=2bs^2, \hspace{5pt} U_0\mp1=2^2act^2,$\\
$8^{\pm}) \enspace\enspace U_0\pm1=2cs^2, \hspace{5pt} U_0\mp1=2^2abt^2.$\\
\end{center}
From (\ref{eq:29}) we get $U_0^2\equiv1\pmod{(\delta k-1)}$ and  $U_0^2\equiv1\pmod{(\delta k-\delta+1)}$, so we assume
\begin{equation}\label{eq:evoevo}
U_0\equiv -1\pmod{(\delta k-1)}, \enspace U_0\equiv1\pmod{(\delta k-\delta+1)}.
\end{equation}
We easily get
$$X_0=2g(2k-1)U_0\equiv d\delta(\delta-2)\pmod{(\delta k-1)}.$$
Because $X_0\equiv d\delta(\delta-2)\pmod{(\delta k-1)(\delta k-\delta+1)}$ and $X_0\equiv d\delta(\delta-2)\equiv0\pmod{2},$ we get
$$X_0\equiv d\delta(\delta-2)\pmod{2(\delta k-1)(\delta k-\delta+1)},$$ so $(\ref{eq:lalala})$ is satisfied. Methods that we use in this section depend on the residue classes modulo $8$ for $\delta$ even, so we deal with each of the four cases separately.\\

\begin{subsection}{$\delta\equiv4\pmod{8}$}
\indent We set $\delta\equiv4\pmod{8}$ and $k\equiv3\pmod{8}$. We obtain

$$a=2k^2-2k+1\equiv5\pmod{8},$$
\vspace{-15pt}
$$b=\delta k-1\equiv3\pmod{8},$$
\vspace{-15pt}
\begin{equation}\label{eq:kongru}
c=\delta k-\delta+1\equiv1\pmod{8}.
\end{equation}

\noindent We prove that there exist infinitely many integers $k$ such that only factorizations $4^-)$ and $7^+)$ are possible. That condition implies that congruence (\ref{eq:lalala}) holds and, consequently, that $(X, Y)$ are integer solutions of (\ref{eq:28}). We deal with each factorization separately. \\

\noindent ${1^+) \enspace U_0+1=2abcs^2, \enspace U_0-1=2^2t^2}$.\\
From $abcs^2-2t^2=1$ we get $7s^2-2t^2\equiv1\pmod{8}$  which does not hold for any $s, t\in\mathbb{Z},$ so this factorization is not possible for $\delta\equiv4\pmod{8}$ and $k\equiv3\pmod{8}$.\\\\
$1^-) \enspace U_0+1=2^2t^2, \enspace U_0-1=2abcs^2$.\\
From $2t^2-abcs^2=1$ we get the congruence $2t^2-7s^2\equiv1\pmod{8}$ that is solvable for $t\equiv0 \pmod{2}$ and $s\equiv1\pmod{2}$. Setting the condition $$\left(\frac{2}{a}\right)=\left(\frac{2}{b}\right)=-1,$$ makes this factorization  impossible.\\\\
$2^+) \enspace U_0+1=2^2abcs^2, \enspace U_0-1=2t^2$.\\
We get $7t^2-2s^2\equiv1\pmod{8}$ which is not satisfied for any $s, t\in\mathbb{Z}$.\\\\
$2^-) \enspace U_0+1=2t^2, \enspace U_0-1=2^2abcs^2$.\\
The equation $t^2-2abcs^2=1$ is a contradiction with minimality of the fundamental solution $(U_0, V_0)$.\\\\
$3^+) \enspace U_0+1=2abs^2, \enspace U_0-1=2^2ct^2$.\\
We get $7s^2-2t^2\equiv1\pmod{8}$ which is not satifsied for any $s, t\in\mathbb{Z}$.\\\\
$3^-) \enspace U_0+1=2^2ct^2, \enspace U_0-1=2abs^2$.\\
The congruence $2t^2-7s^2\equiv1\pmod{8}$ is satisfied for $t\equiv0\pmod{2}$ and $s\equiv1\pmod{2}$, so we set the condition   $\left(\frac{2c}{a}\right)=-1$ or $\left(\frac{2c}{b}\right)=-1$ that makes the congruence $2t^2-7s^2\equiv1\pmod{8}$ unsolvable. \\ \indent We get
$$\left(\frac{2c}{a}\right)=\left(\frac{2}{a}\right)\left(\frac{c}{a}\right)=-\left(\frac{c}{a}\right)=-1 \enspace \Rightarrow \enspace \left(\frac{c}{a}\right)=1,$$
or $$\left(\frac{2c}{b}\right)=\left(\frac{2}{b}\right)\left(\frac{c}{b}\right)=-\left(\frac{c}{b}\right)=-1 \enspace \Rightarrow \enspace \left(\frac{c}{b}\right)=1.$$
$4^+) \enspace U_0+1=2acs^2, \enspace U_0-1=2^2bt^2$.\\
We deal with the congruence $5s^2-6t^2\equiv1\pmod{8}$ which is not satisfied for any $s, t\in\mathbb{Z}$. \\\\
$4^-) \enspace U_0+1=2^2bt^2, \enspace U_0-1=2acs^2$.\\
We get $6t^2-5s^2\equiv1\pmod{8}$ which is satisfied for $s\equiv t\equiv1\pmod{2}$, so we establish a condition that would make this case impossible. For 
$$\left(\frac{2b}{a}\right)=\left(\frac{2}{a}\right)\left(\frac{b}{a}\right)=-\left(\frac{b}{a}\right)=-1 \enspace \Rightarrow \enspace \left(\frac{b}{a}\right)=1,$$
or $$\left(\frac{2b}{c}\right)=\left(\frac{2}{c}\right)\left(\frac{b}{c}\right)=\left(\frac{b}{c}\right)=-1 \enspace \Rightarrow \enspace \left(\frac{b}{c}\right)=-1$$ the congruence is unsolvable.\\\\
$5^+) \enspace U_0+1=2bcs^2, \enspace U_0-1=2^2at^2$.\\
It is obtained $3s^2-2t^2\equiv1\pmod{8}$ which is satisfied for $s\equiv t\equiv1\pmod{2}$. If we want that the congruence does not have any solutions, one of the following conditions has to be satisfied 
$$\left(\frac{-2a}{b}\right)=\left(\frac{-1}{b}\right)\left(\frac{2}{b}\right)\left(\frac{a}{b}\right)=\left(\frac{a}{b}\right) \enspace \Rightarrow \enspace \left(\frac{a}{b}\right)=-1$$
or
$$\left(\frac{-2a}{c}\right)=\left(\frac{-1}{c}\right)\left(\frac{2}{c}\right)\left(\frac{a}{c}\right)=\left(\frac{a}{c}\right) \enspace \Rightarrow \enspace \left(\frac{a}{c}\right)=-1.$$
$5^-) \enspace U_0+1=2^2at^2, \enspace U_0-1=2bcs^2$.\\
The congruence $2t^2-3s^2\equiv1\pmod{8}$ is not satisfied for any $s, t\in\mathbb{Z}$, so this factorization is not possible.\\\\
$6^+) \enspace U_0+1=2as^2, \enspace U_0-1=2^2bct^2$.\\
The congruence $2t^2-3s^2\equiv1\pmod{8}$ is not satisfied for any $s, t\in\mathbb{Z}$.\\\\
$6^-) \enspace U_0+1=2^2bct^2, \enspace U_0-1=2as^2$.\\
From $2bct^2-as^2=1$ we get $6t^2-5s^2\equiv1\pmod{8}$ which is satisfied for $s\equiv t\equiv1\pmod{2}$. Conditions that make this case impossible are
$$\left(\frac{-a}{b}\right)=-1 \enspace \Rightarrow \enspace \left(\frac{a}{b}\right)=1$$
or
$$\left(\frac{-a}{c}\right)=-1 \enspace \Rightarrow \enspace  \left(\frac{a}{c}\right)=-1.$$
$7^+) \enspace U_0+1=2bs^2, \enspace U_0-1=2^2act^2$.\\
From $bs^2-2act^2=1$, we get $3s^2-2t^2\equiv1\pmod{8}$, which is satisfied for $s\equiv t\equiv1\pmod{2}$.
The congruence $3s^2-2t^2\equiv1\pmod{8}$ does not have any solutions in $s$ and $t$ for
$$\left(\frac{b}{a}\right)=-1$$
or
$$\left(\frac{b}{c}\right)=-1.$$
$7^-) \enspace U_0+1=2^2act^2, \enspace U_0-1=2bs^2$.\\
The equation $2act^2-bs^2=1$ implies $2t^2-3s^2\equiv1\pmod{8}$ which is not satisfied for any $s, t\in\mathbb{Z}$.\\\\
$8^+) \enspace U_0+1=2cs^2, \enspace U_0-1=2^2abt^2$.\\
From the equation $cs^2-2abt^2=1$ we get $s^2-6t^2\equiv1\pmod{8}$ which is solvable for $s\equiv1\pmod{2}$ and $t\equiv0\pmod{2}$. The congruence does not have any solutions if one of the following condition is satisfied $$\left(\frac{c}{a}\right)=-1$$
or
$$\left(\frac{c}{b}\right)=-1.$$
$8^-) \enspace U_0+1=2^2abt^2, \enspace U_0-1=2cs^2$.\\
We get $2abt^2-cs^2=1$ that implies $6t^2-s^2\equiv1\pmod{8}$ which is not satisfied for any $s, t\in\mathbb{Z}$.\\\\
\indent From the above observations we notice that factorizations $3^-), 4^-), 5^+), 6^-), 7^+), 8^+)$ are possible. If we set conditions $$\left(\frac{a}{c}\right)=\left(\frac{c}{a}\right)=-1 \enspace\enspace \textnormal{and} \enspace\enspace \left(\frac{c}{b}\right)=\left(\frac{b}{c}\right)=1,$$
the only possible factorizations are $4^-)$ and $7^+)$. For $\left(\frac{b}{a}\right)=\left(\frac{a}{b}\right)=-1$ the only possible case is $4^-)$ and for $\left(\frac{b}{a}\right)=\left(\frac{a}{b}\right)=1$ the only possible case is $7^+)$.\\

We prove conditionally that the determined conditions can be fulfilled if a famous conjecture is satisfied.

Let $k$ be an integer that satisfies the following conditions:\\

\noindent(i)\enspace $k\equiv3\pmod{8}$,\\\\
(ii)\enspace $\left(\frac{\delta k-\delta+1}{A}\right)=-1$ for $A=\frac{\delta^2}{2}-\delta+1$,\\\\
(iii)\enspace $\left(\frac{\delta k-\delta+1}{B}\right)=1$ for $B=\frac{\delta}{2}-1$,\\\\
(iv)\enspace $a=2k^2-2k+1$ is prime,\\\\
(v) \enspace $b=\delta k-1$ is prime,\\\\
(vi)\enspace $c=\delta k-\delta+1$ is prime.\\\\
\indent The condition (i) implies that $a, b, c$ defined in (iv), (v), (vi) are of the form
$$a\equiv5\pmod{8}, \enspace b\equiv3\pmod{8}, \enspace c\equiv1\pmod{8}.$$
We show that the condition (ii) is equivalent to $\left(a/c\right)=-1$ and the condition (iii) is equivalent to $\left(b/c\right)=1$. More precisely, we have
$$\left(\frac{a}{c}\right)=\left(\frac{2k^2-2k+1}{\delta k-\delta+1}\right)=\left(\frac{2\delta^2k^2-2\delta^2k+\delta^2}{\delta k-\delta+1}\right)=\left(\frac{2\delta k(\delta k-\delta+1)-2\delta k+\delta^2}{\delta k-\delta+1}\right)=$$
$$=\left(\frac{-2\delta k+\delta^2}{\delta k-\delta+1}\right)=\left(\frac{-2(\delta k-\delta+1)-2\delta+2+\delta^2}{\delta k-\delta+1}\right)=$$
$$=\left(\frac{\delta^2/2-\delta+1}{\delta k-\delta+1}\right)=\left(\frac{\delta k-\delta+1}{\delta^2/2-\delta+1}\right)=\left(\frac{c}{A}\right),$$ 
for $A=\frac{\delta^2}{2}-\delta+1$.\\
\indent Furthermore, we get\\
$$\left(\frac{c}{b}\right)=\left(\frac{\delta k-\delta+1}{\delta k-1}\right)=\left(\frac{\delta k-1}{\delta k-\delta+1}\right)=\left(\frac{\delta k-\delta+1+\delta-2}{\delta k-\delta+1}\right)=\left(\frac{\delta-2}{\delta k-\delta+1}\right)=$$
$$=\left(\frac{\delta/2-1}{\delta k-\delta+1}\right)=\left(\frac{\delta k-\delta+1}{\delta/2-1}\right)=\left(\frac{c}{B}\right),$$
for $B=\delta/2-1$.\\\\
\indent We check if conditions (i), (ii) and (iii) can be simultaneously satisfied. It can be easily shown that $\textrm{gcd}(AB, \delta)=1$. \\
\indent Let $$A=p_1^{a_1}p_2^{a_2}\cdot\dots\cdot p_r^{a_r}$$ be canonical prime factorization of $A$. We have $A\equiv5\pmod{8}$, so  $A$ is not a perfect square. Furthermore, 
\begin{equation}\label{eq:djelivosta1}
A\not\equiv0\pmod{3}.
\end{equation} Because $A$ is not a perfect square, some of its exponents $a_i$ in the canonical prime factorization are odd. Without a loss of generality, let $a_1$ be odd and let $x_1$ be some quadratic nonresidue modulo $p_1$. Because $p_1\geq5$, there are $(p_1-1)/2\geq2$ quadratic nonresidues modulo $p_1$, so we choose $x_1$ such that $$x_1\not\equiv2-\delta\pmod{p_1}.$$

Because $\textrm{gcd}(A, B)=1$ according to Chinese Remainder Theorem we conclude that there exist infinitely many integers that satisfy the congruences $$x\equiv x_1\pmod{p_1}, \enspace x\equiv 1 \pmod{p_i}, \enspace i=2, \dots r, \ \ \ x\equiv 1\pmod{B}.$$
We define $k$ as $$k=\frac{x+\delta-1}{\delta}.$$
For $k\equiv3\pmod{8}$ we conclude $x$ satisfies $$x=\delta k-\delta+1=\delta(8l+3)-\delta+1\equiv 2\delta+1\pmod{8\delta}.$$
Because $\textrm{gcd}(AB, 8\delta)=1$, the system of congruences
$$x\equiv x_1\pmod{p_1}, \enspace x\equiv1\pmod{p_i}, \enspace i=2, \dots, r,$$
$$x\equiv1\pmod{B}, \enspace x\equiv2\delta+1\pmod{8\delta}$$
is solvable. If $x_0$ is one solution of the above system, then all solutions $x$ are of the form
$$x\equiv x_0\pmod{8p_1\dots p_rB\delta}.$$
Obviously, all the solutions of the mentioned system satisfy the conditions
$$\left(\frac{x}{A}\right)=\left(\frac{x_1}{p_1}\right)=-1 \enspace\enspace \textrm{and} \enspace\enspace \left(\frac{x}{B}\right)=\left(\frac{1}{B}\right)=1,$$
and are of the form $x=\delta k-\delta+1$, $k\equiv3\pmod{8}.$ This shows us the conditions (i), (ii) and (iii) are simultaneously satisfied. 

It remains to answer whether conditions (iv), (v) and (vi) can be simultaneously satisfied while conditions (i), (ii) and (iii) are fulfilled, too. In order to answer that question, we use Schinzel's hypothesis H \cite{schinzel}. 

\begin{theorem}[Schinzel's Hypothesis $H$]\label{eq:slutnja}
	\textit{Suppose $f_1(x), \dots, f_m(x)$ are polynomials with integer coefficients. If the following conditions are satisfied
		\begin{itemize}
			\item $f_i(x)$ are irreducible for $i=1, 2, \dots, m$,
			\item for every prime $p$ there exists a positive integer $n$ for which $$f_1(n)f_2(n)\dots f_m(n)\not\equiv0\pmod{p},$$
			\end{itemize}
		then there exist infinitely many positive integers $t$ such that $$f_1(t), f_2(t), \dots, f_m(t)$$ are simultaneously prime numbers.}
	\end{theorem}

	\begin{prop}\label{prop1}
If Schinzel's Hypothesis H holds, then for all positive integers $\delta\equiv4\pmod{8}$ there exist infinitely many positive odd integers $n$ for which there are two divisors $d_1, d_2>1$ of $\frac{n^2+1}{2}$ such that $$d_1+d_2=\delta n+\delta-2.$$
	\end{prop}
	\begin{proof}
	\indent We have already shown that positive integers $k$ we have earlier defined simultaneously satisfy conditions (i), (ii) and (iii). We show there exist infinitely many positive integers  $k$ for which 
	\begin{equation}\label{eq:deff}
	a=2k^2-2k+1,\enspace b=\delta k-1 \enspace \textrm{and} \enspace c=\delta k-\delta+1
	\end{equation}
	are simultaneously prime. We assume 
	\begin{equation}\label{eq:defk}
	k\equiv y_0\pmod{(8p_1p_2\dots p_rB)} \enspace \textrm{or} \enspace k=8p_1p_2\dots p_rBe+y_0, \enspace e\in\mathbb{N}.\end{equation}
	Let $s=8p_1p_2\dots p_rB$ or $k= se+y_0$. We deal with polynomials of the form 
	\begin{equation}\label{eq:deff}
	f_1(k)=2k^2-2k+1,\enspace f_2(k)=\delta k-1, \enspace f_3(k)=\delta k-\delta+1
	\end{equation} that in our case become polynomials 
	$$f_1(e)=2s^2e^2+2s(2y_0-1)e+2y_0^2-2y_0+1,$$
	$$f_2(e)=\delta se+\delta y_0-1,$$
	$$f_3(e)=\delta se+\delta y_0-\delta+1.$$
	We prove $f_1(e), f_2(e), f_3(e)$ satisfy conditions of Schinzel's Hypothesis H. Polynomials $f_1, f_2, f_3$ are irreducible so they satisfy the first condition of Schinzel's hypothesis H. \\
	\indent Now we prove that for every prime number $p$ there exists a positive integer $n$ for which $$f_1(n)f_2(n)f_3(n)\not\equiv0\pmod{p}.$$
	We deal with three cases $p=2,$ $p=3$ and $p\geq5$, $p$ prime.\\
	\indent Because $\delta$ is even, for $p=2$ we have $f_1(e)\equiv f_2(e)\equiv f_3(e)\equiv1\pmod{2},$ so we conclude that for every positive integer $e$ we have $$f_1(e)f_2(e)f_3(e)\not\equiv0\pmod{2}.$$
	For $p=2$ the second condition of Schinzel's Hypothesis H is satisfied.
	
 Let $p=3$. We show $f_1(e)\not\equiv0\pmod{3}$ for every positive integer $e$. Indeed, if the congruence $f_1(e)\equiv0\pmod{3}$ is satisfied, then $$2f_1(e)\equiv(2se+(2y_0-1))^2+1\equiv0\pmod{3}$$
	would imply $\left(\frac{-1}{3}\right)=1$ which is a contradiction.\\
	\indent We distinguish two cases: $3\mid s$ or $3\nmid s$. For $3\mid s$ the congruence (\ref{eq:djelivosta1}) implies that $3\nmid A$ or $3\mid B$ which implies $\delta\equiv2\pmod{3}$. From $x\equiv1\pmod{B}$ we have $x\equiv1\pmod{3}.$ On the other side, because \begin{equation}\label{eq:sch1}
	x=\delta k-\delta+1,
	\end{equation} we have $$x\equiv\delta y_0-\delta+1\equiv 2y_0-1\equiv1\pmod{3}.$$ Consequently, $3\nmid(\delta y_0-1)$ and $3\nmid(\delta y_0-\delta+1),$ so congruences $f_2(e)\equiv0\pmod{3}$ and $f_3(e)\equiv0\pmod{3}$ are unsolvable. \\
	\indent If $3\nmid s$, then congruences $f_2(e)\equiv0\pmod{3}$ and $f_3(e)\equiv0\pmod{3}$ have at most one solution modulo $3$. But, this means that there exists at least one residue class modulo $3$ that does not satisfy any of two mentioned congruences so there are infinitey many positive integers $e$ that satisfy the congruence $$f_1(e)f_2(e)f_3(e)\not\equiv0\pmod{3}.$$
	The second condition of Schinzel's Hypothesis H is satisfied for $p=3$.\\ 
	\indent Now, let $p\geq5$ a prime. Again, we distinguish two cases: $p\mid s$ or $p\nmid s$.
	If $p\mid s$, then $p\mid A$ or $p\mid B$. We have $$f_1(e)\equiv2y_0^2-2y_0+1\pmod{p},$$
	$$f_2(e)\equiv\delta y_0-1\pmod{p},$$
	$$f_3(e)\equiv \delta y_0-\delta+1\pmod{p}.$$
	\indent If $p\mid B$, then $\delta\equiv2\pmod{p}$, so from $x\equiv 1\pmod{p}$ and (\ref{eq:sch1}) we easily get $$x\equiv\delta y_0-\delta+1\equiv 2y_0-1\equiv1\pmod{p}.$$ We conclude $y_0\equiv1\pmod{p}.$ So, we have $$2y_0^2-2y_0+1\equiv1\pmod{p}, \enspace \delta y_0-1\equiv1\pmod{p}, \enspace \delta y_0-\delta+1\equiv1\pmod{p}$$ and congruences $f_1(e)\equiv0\pmod{p}$, $f_2(e)\equiv0\pmod{p}$, $f_3(e)\equiv0\pmod{p}$ do not have solutions.\\
	\indent If $p\mid A$, we distinguish two cases: $p=p_1$ or $p=p_i$ for $i\in\{2, \dots, r\}$. Let $p=p_i$ for $i\in\{2, \dots, r\}$. From (\ref{eq:sch1}) we get  $$x\equiv \delta y_0-\delta +1\equiv1\pmod{p_i},$$ so we have $\delta(y_0-1)\equiv0\pmod{p_i}.$ From $2A=\delta(\delta-2)+2=(\delta-1)^2+1\equiv0\pmod{p_i}$ we get $p_i\nmid \delta$ and $p_i\nmid (\delta-1)$ so we have $y_0\equiv1\pmod{p_i}.$ Because of $$2y_0^2-2y_0+1\equiv1\not\equiv0\pmod{p_i}, \enspace \delta y_0-\delta+1\equiv1\not\equiv0\pmod{p_i},$$ $$\delta y_0-1\equiv\delta-1\not\equiv0\pmod{p_i},$$ congruences $f_1(e)\equiv0\pmod{p_i},$ $f_2(e)\equiv0\pmod{p_i},$ $f_3(e)\equiv0\pmod{p_i}$ do not have solutions.\\
	\indent Finally, let $p=p_1$. From (\ref{eq:sch1}) we have
	\begin{equation}\label{eq:poziv1}
	x\equiv\delta y_0-\delta+1\equiv x_1\pmod{p_1},
	\end{equation}
	where $x_1$ is quadratic nonresidue modulo $p_1$ and $x_1\not\equiv2-\delta\pmod{p_1}$ is satisfied. So, we have that congruences $f_2(e)\equiv0\pmod{p_1}$ and $f_3(e)\equiv0\pmod{p_1}$ do not have solutions. It remains to deal with the congruence $f_1(e)\equiv0\pmod{p_1},$ or more precisely with $$2y_0^2-2y_0+1\equiv0\pmod{p_1}.$$
	From (\ref{eq:sch1}) and (\ref{eq:poziv1}) we have $$\delta^2(2y_0^2-2y_0+1)\equiv 2x_1^2+2x_1\delta-4x_1+\delta^2-2\delta+2\equiv2x_1(x_1+\delta-2)\not\equiv0\pmod{p_1}$$ so the congruence $f_1(e)\equiv0\pmod{p_1}$  does not have solutions.\\
	\indent If $p\nmid s$, the congruence $f_1(e)\equiv0\pmod{p}$ has at most two solutions modulo $p$, while congruences $f_2(e)\equiv0\pmod{p}$ and $f_3(e)\equiv0\pmod{p}$ have at most one solution modulo $p$. Hence, there exists at least one residue class modulo $p$ which does not satisfy any of three mentioned congruences. There are infinitely many positive integers $e$ that satisfy $$f_1(e)f_2(e)f_3(e)\not\equiv0\pmod{p}.$$
	\end{proof}

	\begin{exmp}
	\normalfont	For $\delta=12$ we get $A=61$, $B=5$. 
		We exclude $x_1\equiv51\pmod{61}.$ 
		The chosen system of congruences is $$x\equiv24\pmod{61}, \enspace x\equiv1\pmod{5}, \enspace x\equiv25\pmod{96}.$$
		Solutions of the above system of congruences are $$x\equiv16921\pmod{29280}.$$
		We write $x=29280s+16921,\enspace s\in\mathbb{Z}$.
		We get $k\equiv1411\pmod{2440}$ or $$k=2440u+1411, \enspace u\in\mathbb{Z}.$$ Introducing $k$ into (iv), (v) and (vi) three polynomials are defined:
		$$a=f_1(u)=11907200u^2 + 13766480 u + 3979021,$$ 
		$$b=f_2(u)=29280 u+16931,$$
		$$c=f_3(u)=29280 u+16921.$$
		The first condition of the Schinzel's Hypothesis is satisfied.  We give the explicit check of the second condition of Schinzel' Hypothesis H and show it is enough to set $n=1, 2, 3$ in order to check that the second condition of the Schinzel's Hypothesis H is satisfied.\\
		For $n=1$ we get
		$$f_1(1)\cdot f_2(1)\cdot f_3(1)=(13\cdot2280977)\cdot(11\cdot4201)\cdot(47\cdot983).$$
		For $n=2$ it is obtained
		$$f_1(2)\cdot f_2(2)\cdot f_3(2)=79140781\cdot(13\cdot5807)\cdot(7\cdot41\cdot263),$$
		while for $n=3$ we get
		$$f_1(3)\cdot f_2(3)\cdot f_3(3)=(641\cdot237821)\cdot(17\cdot6163)\cdot104761.$$
		We notice $$\textrm{gcd}(f_1(1)\cdot f_2(1)\cdot f_3(1),\enspace f_1(2)\cdot f_2(2)\cdot f_3(2), f_1(3)\cdot f_2(3)\cdot f_3(3))=1,$$
		so we have shown that prime $p$ that divides three products $f_1(n)f_2(n)f_3(n), \enspace n=1, 2, 3$ does not exist. \\
		
		For $k\leq10^9$ there are $153$ positive integers $k$ that satisfy mentioned conditions. First of such positive integers $k$ are:
		$$1411, 16051, 240531, 360091, 425971, 626051, 1314131, 1975371, 2241331, 2426771, 2495091 \dots$$
		For $k=1411$, the appropriate Pell equation and it's solutions are:\\
		$$U^2-2279895083614942V^2=1,$$
		
		$$U_0\approx2.58023\cdot10^{1502988},\enspace V_0\approx1.54982\cdot10^{1502980}.$$
We have $X_0=2g(2k-1)U_0$ where $g=\frac{\delta^2-2\delta+2}{2}$. Finally, from
		$$n=\frac{X_0-d\delta(\delta-2)}{d\delta^2-2g},$$
	we get	$$n\approx1.54982\cdot10^{1502985},$$
		while divisors of $(n^2+1)/2$ are
		$$d_1\approx9.89977\cdot10^{1502978},\enspace d_2\approx1.85979\cdot10^{1502986}.$$
	\end{exmp}
\end{subsection}

\begin{subsection}{$\delta\equiv6\pmod{8}$}
\indent Let $\delta\equiv6\pmod{8}$ and $k\equiv2\pmod{8}$. For integers $a, b, c$ we get
$$a=2k^2-2k+1\equiv5\pmod{8},$$
\vspace{-20pt}
$$b=\delta k-1\equiv3\pmod{8},$$
\vspace{-15pt}
\begin{equation}\label{eq:kongr2}
c=\delta k-\delta+1\equiv7\pmod{8}.
\end{equation}
Like in the previous subsection, we deal with the following cases:\\
$1^+) \enspace U_0+1=2abcs^2, \enspace U_0-1=2^2t^2$.\\
Because $a\equiv5\pmod{8}$, we have
$$\left(\frac{-2}{a}\right)=\left(\frac{-1}{a}\right)\left(\frac{2}{a}\right)=-1$$ which makes this case impossible.\\\\
\noindent$1^-) \enspace U_0+1=2^2t^2, \enspace U_0-1=2abcs^2$.\\
For $a\equiv5\pmod{8}$ and $b\equiv3\pmod{8}$ this case is not possible.\\\\
$2^+) \enspace U_0+1=2^2abcs^2, \enspace U_0-1=2t^2$.\\
Because $b\equiv3\pmod{8},$ we get $\left(\frac{-1}{b}\right)=-1$ which makes this case impossible.\\\\
$2^-) \enspace U_0+1=2t^2, \enspace U_0-1=2^2abcs^2$.\\
We get $t^2-2abcs^2=1,$ which is a contradiction with minimality of $(U_0, V_0)$.\\\\
$3^+) \enspace U_0+1=2abs^2, \enspace U_0-1=2^2ct^2$.\\
If we want this case to be impossible, then it should be satisfied one of the following conditions $$\left(\frac{-2c}{a}\right)=\left(\frac{-1}{a}\right)\left(\frac{2}{a}\right)\left(\frac{c}{a}\right)=-\left(\frac{c}{a}\right)=-1 \enspace \enspace \Rightarrow \enspace \left(\frac{c}{a}\right)=1,$$ or $$\left(\frac{-2c}{b}\right)=\left(\frac{-1}{b}\right)\left(\frac{2}{b}\right)\left(\frac{c}{b}\right)=\left(\frac{c}{b}\right)=-1 \enspace \enspace \Rightarrow \enspace  \left(\frac{c}{b}\right)=-1.$$
$3^-) \enspace U_0+1=2^2ct^2, \enspace U_0-1=2abs^2$.\\
If we want this case to be impossible, then one of the following conditions has to be fulfilled 
$$\left(\frac{2c}{a}\right)=\left(\frac{2}{a}\right)\left(\frac{c}{a}\right)=-\left(\frac{c}{a}\right)=-1 \enspace \Rightarrow \enspace \left(\frac{c}{a}\right)=1,$$
or
$$\left(\frac{2c}{b}\right)=\left(\frac{2}{b}\right)\left(\frac{c}{b}\right)=-\left(\frac{c}{b}\right)=-1 \enspace \Rightarrow \enspace \left(\frac{c}{b}\right)=1.$$

\noindent Cases $4^+), 4^-), 5^+), 5^-), 6^+), 6^-), 7^+), 7^-)$ are impossible in for integers $a, b, c$ that satisfy congruences (\ref{eq:kongr2}).\\


\noindent $8^+) \enspace U_0+1=2cs^2, \enspace U_0-1=2^2abt^2$.\\
If we want that the equation does not have any solutions, then one of the following conditions  
$$\left(\frac{c}{a}\right)=-1$$
or
$$\left(\frac{c}{b}\right)=-1$$ has to be satisfied.\\\\
$8^-) \enspace U_0+1=2^2abt^2, \enspace U_0-1=2cs^2$.\\
One of the following conditions 
$$\left(\frac{-c}{a}\right)=\left(\frac{-1}{a}\right)\left(\frac{c}{a}\right)=\left(\frac{c}{a}\right)=-1 \enspace \Rightarrow \enspace \left(\frac{c}{a}\right)=-1$$
or
$$\left(\frac{-c}{b}\right)=\left(\frac{-1}{b}\right)\left(\frac{c}{b}\right)=-\left(\frac{c}{b}\right)=-1 \enspace \Rightarrow \enspace \left(\frac{c}{b}\right)=1$$ 
has to be fulfilled if we want that this case is not possible.\\\\
\indent Possible factorizations are $3^+), 3^-), 8^+), 8^-)$. If we set conditions $$\left(\frac{c}{a}\right)=\left(\frac{a}{c}\right)=1, \enspace \left(\frac{c}{b}\right)=\left(\frac{b}{c}\right)=-1,$$
we get that the only possible factorization is $8^-)$.

\indent For $\delta\equiv6\pmod{8}$ and $k\equiv2\pmod{8}$ we have
$$\left(\frac{c}{a}\right)=
\left(\frac{-c}{A}\right),$$
for $A=\delta^2/2-\delta+1$. We easily determine $A\equiv1\pmod{4}$, where $A$ can be composite.\\
We also obtain
$$\left(\frac{c}{b}\right)=$$\[=\left\{ 
\begin{array}{l l}
\left(\frac{\delta k-\delta+1}{B}\right)=\left(\frac{c}{B}\right)=-\left(\frac{-c}{B}\right), & \quad \text{for $\delta\equiv14\pmod{16}, \enspace B\equiv3\pmod{4}$}\\\\
-\left(\frac{\delta k-\delta+1}{B}\right)=-\left(\frac{c}{B}\right)=-\left(\frac{-c}{B}\right), & \quad \text{for $\delta\equiv6\pmod{16}, \enspace B\equiv1\pmod{4}$}
\end{array} \right.\]
for $B=(\delta-2)/4$, where $B$ can be composite, too.\\

\indent Let $k$ be such that the following conditions are fulfilled: \\\\
\noindent(i)\enspace\enspace $k\equiv2\pmod{8}$,\\\\
(ii)\enspace\enspace $\left(\frac{-c}{A}\right)=1$ for $A=\frac{\delta^2}{2}-\delta+1$, \\\\
(iii)\enspace\enspace $\left(\frac{-c}{B}\right)=1$ for $B=(\delta-2)/4$,\\\\
(iv)\enspace\enspace $a=2k^2-2k+1$ is prime,\\\\
(v)\enspace\enspace $b=\delta k-1$ is prime,\\\\
(vi)\enspace\enspace $c=\delta k-\delta+1$ is prime.\\\\
We have already shown that the condition (ii) is equivalent with the condition $(c/a)=1$ and the condition (iii) is equivalent with $(c/b)=-1.$ Now we check if the conditions (i), (ii) and (iii) can be simultaneously satisfied. 

\indent We easily obtain $\textrm{gcd}(A, B)=\textrm{gcd}(A, \delta)=1$. Because $4B=\delta-2$ and $B$ odd, we have $\textrm{gcd}(B, \delta)=1$ which implies $\textrm{gcd}(AB, \delta)=1$. 
 
Using Chinese remainder theorem we conclude that there exist infinitely many integers $x$ that satisfy the following system of congruences
$$x\equiv x_i\pmod{p_i},  \enspace x\equiv1\pmod{B}, \enspace i=1, 2, \dots, r,$$ where $x_i$ is a quadratic residue modulo $p_i$, where $p_i, \  i=1, 2,\dots, r$ are all different prime factors of $A$. 
We get
\begin{equation}\label{eq:djelivostb}
A=\frac{\delta^2}{2}-\delta+1\not\equiv0\pmod{3}.
\end{equation} 
Like in the previous section, for every prime factor $p_i$ of $A$ we have $p_i\geq5$ and there are $(p_i-1)/2\geq2$ quadratic residues modulo $p_i$, so we choose $x_i$ such that $$x_i\not\equiv\delta-2\pmod{p_i}.$$

\indent We define $x=-(\delta k-\delta+1)=\delta-\delta k-1$ or $$k=\frac{\delta-x-1}{\delta},$$
where $x$ is a solution of the above system of the congruences. Let $k\equiv2\pmod{8}$. 
In this case $x$ satisfies
$$x
\equiv-\delta-1\pmod{8\delta}.$$
Because $\textrm{gcd}(AB, 8\delta)=1$, we get that the system of the congruences
$$x\equiv x_i\pmod{p_i}, \enspace x\equiv1\pmod{B}, \enspace x\equiv -\delta-1\pmod{8\delta}, \enspace i=1, 2,\dots, r,$$
has solutions. If $x_0$ is one of its solutions, all solutions $x$ are of the form
$$x\equiv x_0\pmod{8AB\delta}.$$
Obviously, all solutions $x$ satisfy  $$\left(\frac{x}{A}\right)=\left(\frac{x_i}{A}\right)=1 \enspace \enspace \textrm{and} \enspace \enspace \left(\frac{x}{B}\right)=\left(\frac{1}{B}\right)=1,\enspace i=1, 2, \dots, r,$$
as well as solutions of the form $x=\delta k-\delta-1,$ where $k\equiv2\pmod{8}$. Hence, conditions (i), (ii) and (iii) are simultaneously satisfied.
\begin{prop}\label{eq:predzadnja}
	\textit{If Schinzel's Hypothesis H is satisfied, then for all positive integers $\delta\equiv6\pmod{8}$ there are infinitely many positive odd integers $n$ such that there exist divisors $d_1, d_2$ of $\frac{n^2+1}{2}$ for which $d_1+d_2=\delta n+\delta-2$.}
\end{prop}
\begin{proof}
We assume (\ref{eq:defk}). 
We have already shown that the polynomials $f_1(e), f_2(e), f_3(e)$ are irreducible, so it only remains to show that the second condition of Schinzel's Hypothesis H is satisfied. The main part of the proof is done exactly as in the proof of Proposition \ref{prop1}. 

The proof is separated into three cases: $p=2$, $p=3$ and $p\geq5,$ $p$ prime. Cases for $p=2, 3$ are completely analogous as in Proposition \ref{prop1}.
and for $p\geq5,$ $p$ prime, again we have two cases: $p\mid s$ or $p\nmid s$.\\
\indent  Let $p\mid s$. In this case $p\mid A$ or $p\mid B$. We deal with congruences of the form
$$f_1(e)\equiv2y_0^2-2y_0+1\pmod{p}, \enspace f_2(e)\equiv\delta y_0-1\pmod{p}, \enspace f_3(e)\equiv\delta y_0-\delta+1\pmod{p}.$$
\indent For $p\mid B$, we get 
$$x\equiv\delta-\delta y_0-1\equiv1-\delta y_0\equiv1-2y_0\equiv1\pmod{p}.$$ We have $y_0\equiv0\pmod{p}$ and congruences $f_1(e)\equiv0\pmod{p}, \enspace f_2(e)\equiv0\pmod{p}$ and $f_3(e)\equiv0\pmod{p}$ do not have solutions.\\
\indent For $p\mid A$, we have $p=p_i$ for some $i=1,\dots, r$ 
and we easily get
\begin{equation}\label{eq:poziv2}
x\equiv \delta-\delta y_0-1\equiv x_i\pmod{p}.
\end{equation}
Because $x_i\not\equiv\delta-2\pmod{p},$ we have $1-\delta y_0\equiv x_i-\delta+2\not\equiv0\pmod{p}$ so congruences $f_2(e)\equiv0\pmod{p}$ and $f_3(e)\equiv0\pmod{p}$ do not have solutions. Finally, we deal with the congruence $f_1(e)\equiv0\pmod{p},$ or $$2y_0^2-2y_0+1\equiv0\pmod{p}.$$
Analogously as in Proposition \ref{prop1} we get
$$\delta^2(2y_0^2-2y_0+1)\equiv2x_i^2-2x_i\delta+4x_i+\delta^2-2\delta+2\equiv2x_i(x_i-\delta+2)\pmod{p}$$ so $f_1(e)\equiv0\pmod{p}$ does not have any solutions.\\

\indent For $p\nmid s$, the procedure and conclusions remain the same as in Proposition \ref{prop1}. 
So, polynomials $f_1, f_2, f_3$ satisfy the second condition of Schinzel's Hypothesis H.
\end{proof}

\begin{exmp}
\normalfont	For $\delta=14$ we get $A=85, \enspace B=3$. We exclude  $x_1\equiv2\pmod{5}$ and $x_2\equiv12\pmod{17}.$ So, let $x_1\equiv1\pmod{5}$ and $x_2\equiv1\pmod{17}.$ The system of the congruences we deal with is
	$$x\equiv1\pmod{5}, \enspace x\equiv1\pmod{17}, \enspace x\equiv 1\pmod{3}, \enspace x\equiv-15\pmod{112}.$$
	Solutions of the above system are $$x\equiv8161\pmod{28560}.$$
	Polynomials $f_1, f_2, f_3$ are of the form:
	$$a=f_1(u)=8323200u^2-4753200u+678613,$$
	$$b=f_2(u)=28560u-8149,$$
	$$c=f_3(u)=28560u-8161.$$
	It is shown that the first condition of Schinzel's Hypothesis H is satisfied. We explicitly show that the second condition of the hypothesis is satisfied, too.\\
	For $n=1$ we get
	$$f_1(1)\cdot f_2(1)\cdot f_3(1)=(181\cdot23473)\cdot(20411)\cdot(20399).$$
	For $n=2$ we get 
	$$f_1(2)\cdot f_2(2)\cdot f_3(2)=24465013\cdot(13\cdot3767)\cdot(173\cdot283).$$
	Obviously, it is obtained
	$$\textrm{gcd}(f_1(1)\cdot f_2(1)\cdot f_3(1), f_1(2)\cdot f_2(2)\cdot f_3(2))=1,$$
	so the second condition of Schinzel's hypothesis H is satisfied. Hence, there exist infinitely many positive integers $u$ such that $f_1(u), \enspace f_2(u), \enspace f_3(u)$
	are simultaneously prime.\\
For	$$x\equiv16\pmod{85}, \enspace x\equiv 1\pmod{3}, \enspace x\equiv-15\pmod{112},$$
	we get $$x\equiv28321\pmod{28560}, \ k\equiv-2022\equiv18\pmod{2040}.$$
For $k=18$ we get $$a=613, \enspace b=251, \enspace c=239.$$
	The appropriate Pell equation is $$U^2-73546514V^2=1,$$
	while $$U_0\approx2.91573\cdot10^{691}, \ V_0\approx3.39990\cdot10^{687}.$$
	Finally, we get $$n\approx 1.44598\cdot10^{690},$$
for which divisors $d_1, d_2$ of $(n^2+1)/2$ are
$$d_1\approx7.16336\cdot10^{689} , \enspace d_2\approx2.02366\cdot10^{692}$$
\end{exmp}
	\end{subsection}

\begin{section}{Open problems}
\begin{rmk}
\normalfont Let $\delta\equiv0\pmod{8}$. For $k\equiv3\pmod{8}$ we get $$a\equiv5\pmod{8}, \enspace b\equiv7\pmod{8}, \enspace c\equiv1\pmod{8}.$$
In this case we are not able to eliminate the factorization $5^-)$, or more precisely the case $$U_0+1=2^2at^2, \enspace U_0-1=2bcs^2$$ 
which implies $$U_0\equiv1\pmod{(\delta k -1)} \enspace\enspace \textrm{and} \enspace\enspace U_0\equiv1\pmod{(\delta k-\delta+1)},$$
which is in a contradiction with (\ref{eq:evoevo}).
For example, let $\delta=8$. For the case $5^-)$ we obtain $$\left(\frac{a}{b}\right)=\left(\frac{2k^2-2k+1}{8k-1}\right)=\left(\frac{8k^2-8k+4}{8k-1}\right)=\left(\frac{-7k+4}{8k-1}\right)=$$
$$=\left(\frac{-112k+64}{8k-1}\right)=\left(\frac{50}{8k-1}\right)=\left(\frac{2}{8k-1}\right)=1,$$
$$\left(\frac{a}{c}\right)=\left(\frac{2k^2-2k+1}{8k-7}\right)=\left(\frac{8k^2-8k+4}{8k-7}\right)=\left(\frac{-k+4}{8k-7}\right)=
$$
$$=\left(\frac{-16k+64}{8k-7}\right)=\left(\frac{50}{8k-7}\right)=\left(\frac{2}{8k-7}\right)=1.$$
Furthermore, even though we are able to find solutions for relatively small $\delta$, we cannot determine solutions for $\delta=40$, so we are not sure whether there exist infinitely many positive and odd integers $n$ for which there exist divisors $d_1, d_2>1$ of $(n^2+1)/2$ such that $$d_1 + d_2=\delta n +\delta - 2, \ \delta\equiv0\pmod{8}.$$ 
	\end{rmk}

\begin{rmk}
\normalfont Let $\delta\equiv2\pmod{8}$. If we apply the analogous method presented in cases $\delta\equiv4,6\pmod{8}$, we get more complicated conditions on Legendre's symbols. For example, we get $$\left(\frac{a}{c}\right)-1, \enspace \left(\frac{a}{b}\right)=1, \enspace \left(\frac{c}{b}\right)=1,$$
so we cannot use Chinese remainder theorem and Schinzel's Hypothesis H in order to get similar conclusions.	
\end{rmk}

\end{section}

\section*{Acknowledgement}
	We would like to thank Professor Andrej Dujella for many valuable suggestions and a great help with the preparation of this article.

\end{document}